\documentclass[a4paper,11pt,oneside]{amsart}
\usepackage[english]{babel} 
\usepackage[utf8]{inputenc} 
\usepackage[T1]{fontenc} 

\usepackage{amssymb,amsmath,amsfonts,amsxtra,amsthm} 
\usepackage{stmaryrd} 
\usepackage{mathtools} 
\usepackage{mathrsfs} 
\usepackage{cases} 
\usepackage[all]{xy} 
\usepackage{soul}
\usepackage{framed}
\usepackage{overpic}

\usepackage{geometry} 
\usepackage{lmodern} 
\usepackage{xargs} 
\usepackage{graphicx} 
\usepackage[svgnames,table]{xcolor} 
\usepackage[colorlinks=true,citecolor=Magenta,linkcolor=Blue, urlcolor=Green, pdftoolbar=true]{hyperref} 
\usepackage{enumitem}
\usepackage{etaremune} 
  \usepackage{bm}
\usepackage{marvosym}
\usepackage{footmisc} 

\newcommand{\ensemblenombre}[1]{\ensuremath{\mathbb{#1}}}

\newcommand{\Z}{\ensemblenombre{Z}}

\newcommand{\R}{\ensemblenombre{R}}
\newcommand{\C}{\ensemblenombre{C}}

\newcommand{\abs}[1]{\ensuremath{\left\lvert#1\right\rvert}}
\newcommand{\norme}[1]{\ensuremath{\left\lVert#1\right\rVert}}
\newcommand{\enstq}[2]{\ensuremath{\left\{#1\mathrel{}\middle|\mathrel{}#2\right\}}}

\newcommand{\transp}[1]{\prescript{t}{}{#1}}

\newcommand{\restreinta}{\ensuremath{\mathclose{}|\mathopen{}}}

\newcommand{\piun}[1]{\pi_1({#1})}

\newcommand{\RP}[1]{\ensuremath{\R\mathbf{P}^{#1}}}
\newcommand{\Sn}[1]{\mathbf{S}^{#1}}

\newcommand{\Diff}[2]{\mathopen{}\mathrm{D}_{#1}#2}

\newcommand{\Fitan}[1]{\ensuremath{\mathrm{T}#1}}

\newcommand{\Fiunitan}[1]{\ensuremath{{\mathrm{T}}^1{#1}}}

\newcommand{\PGL}[1]{\ensuremath{\mathrm{PGL}_{#1}\mathopen{(}\R\mathclose{)}}}

\newcommand{\slR}[1]{\ensuremath{\mathfrak{sl}_{#1}}}

\newcommand{\Heis}[1]{\ensuremath{\mathrm{Heis}\mathopen{(}#1\mathclose{)}}}

\newcommand{\PBio}[2]{\ensuremath{\mathrm{PO}\mathopen{(}#1\mathpunct,#2\mathclose{)}}}
\newcommand{\PBiu}[2]{\ensuremath{\mathrm{PU}\mathopen{(}#1\mathpunct,#2\mathclose{)}}}

\newcommand{\X}{\mathbf{X}}

\newcommand{\Lm}{\mathcal{L}}

\newcommand{\Calpha}{\mathcal{C}_\alpha}
\newcommand{\Cbeta}{\mathcal{C}_\beta}

\newcommand{\Falpha}{\mathcal{F}_\alpha}

\newcommand{\Fbeta}{\mathcal{F}_\beta}

\newcommand{\Bplus}{B_{\alpha\beta}^+}
\newcommand{\Bmoins}{B_{\alpha\beta}^-}
\newcommand{\Balphabeta}{B_{\alpha\beta}}

\newcommand{\Pmin}{\mathbf{P}_{min}}

\DeclareMathOperator{\Image}{Im}

\DeclareMathOperator{\Int}{Int}

\DeclareMathOperator{\tr}{tr}
\DeclareMathOperator{\id}{id}

\DeclareMathOperator{\Diag}{Diag}

\numberwithin{equation}{section}
\numberwithin{figure}{section}

\AtBeginDocument{}

\addto\captionsenglish{}

\usepackage{tikz}
\usepackage{tikzpagenodes}
\usetikzlibrary{tikzmark}

\makeatletter
\newcommand{\getchbarstartendpages}[1]{%
    \edef\chbarstartmarkid{\csname save@pt@chbar-#1-start\endcsname}%
    \edef\chbarstartpage{\csname save@pg@\chbarstartmarkid\endcsname}%
    \edef\chbarendmarkid{\csname save@pt@chbar-#1-end\endcsname}%
    \edef\chbarendpage{\csname save@pg@\chbarendmarkid\endcsname}%
}
\makeatother

\newcommand{\chbarifsplit}[3]{%
    \getchbarstartendpages{#1}%
    \ifnum\chbarstartpage=\chbarendpage\relax%
        {#3}%
    \else%
        {#2}%
    \fi%
}

\newcounter{changebar}

\theoremstyle{definition}
\newtheorem{definition}{Definition}[section]

\theoremstyle{plain}
\newtheorem{theorem}[definition]{Theorem}

\newtheorem{theoremintro}{Theorem}

\newtheorem{corollary}[definition]{Corollary}
\newtheorem{corollaryintro}[theoremintro]{Corollary}

\newtheorem{lemma}[definition]{Lemma}
\newtheorem{proposition}[definition]{Proposition}

\theoremstyle{remark}
\newtheorem{example}[definition]{Example}

\newtheorem{remark}[definition]{Remark}

\title[Geometric surgeries of three-dimensional flag structures]
{Geometric surgeries of three-dimensional flag structures and
non-uniformizable examples}
\author{Elisha Falbel and Martin Mion-Mouton}
\date{\today}

\begin{document}
\address{
Martin Mion-Mouton,
Institut de Mathématiques de Marseille (I2M), UMR 7373
}
\email{\href{mailto:martin.mion-mouton@univ-amu.fr}{martin.mion-mouton@univ-amu.fr}}
\urladdr{\url{https://orcid.org/0000-0002-8814-0918}}
\address{
Institut de Math\'ematiques de Jussieu-Paris Rive Gauche,
CNRS UMR 7586 and INRIA EPI-OURAGAN
}
\email{\href{mailto:elisha.falbel@imj-prg.fr}{elisha.falbel@imj-prg.fr}}
\urladdr{\url{https://webusers.imj-prg.fr/~elisha.falbel/}}


\begin{abstract}
In this paper, we introduce a notion of geometric surgery
for flag structures,
which are geometric structures locally modelled on the three-dimensional flag space
under the action of $\mathrm{PGL}_3(\mathbb{R})$.
Using such surgeries we provide examples of flag structures,
of both uniformizable and non-uniformizable type.
\end{abstract}

\maketitle

\section{Introduction}
A fundamental problem in the study of any
locally homogeneous geometric structure
is simply to construct examples of such
structures, and a basic way to produce new examples is to combine
formerly known ones.
With this goal in mind,
an easy way to topologically combine two manifolds
is to form their connected sum,
which raises then the natural question wether the
connected sum of two geometric manifolds
can be endowed with a geometric structure combining the ones of
the two pieces.
In the case of
flat conformal Riemannian manifolds for instance,
locally modelled on the round sphere $\Sn{n}$ with the conformal action of
$\PBio{1}{n+1}$,
or in the one of spherical CR-manifolds modelled on
$\partial\mathbf{H}^n_{\C}$ with the CR action of $\PBiu{1}{n}$,
previous works established such \emph{geometric connected sums}
(see \cite{kulkarni_principle_1978} for the former
and \cite{Burns-Shnider, falbel_non-embeddable_1992} for the latter).
\par However, both of these structures
share the important common feature that they are modelled on
\emph{rank one} simple Lie groups.
Our goal in this paper is to introduce
such a notion of \emph{geometric surgery}
for three-dimensional \emph{flag structures}, which are one of the simplest higher-rank geometries
and are modelled on the three-dimensional flag space under the action of $\PGL{3}$.
Flag structure surgeries will involve  gluings along genus two surfaces.
We then initiate the study of these surgeries and use them to produce
new examples of flag structures.

\subsection{Flag structures in dimension three}
\label{subsection-flagstructuresdimensionthree}
 We will be interested in this paper in a three-dimensional homogeneous space
of the Lie group $\PGL{3}$.
Denoting by $\RP{2}_*$ the space
of projective lines of $\RP{2}$,
the \emph{flag space} is the
set
\begin{equation}\label{equation-flagspace}
 \X=\enstq{(p,D)}{p\in D}\subset\RP{2}\times\RP{2}_*
\end{equation}
of pointed projective lines of $\RP{2}$.
This is a closed orientable three-manifold,
endowed with a transitive diagonal (projective) action of $\PGL{3}$.
We will study in this paper three-manifolds that are
locally modelled on $\X$. 
We will also call \emph{flag structure} a $(\PGL{3},\X)$-structure (see Definition \ref{definition-flagstructure}),
and \emph{flag manifold} a three-manifold endowed with a flag structure.
\par Flag structures can also be defined as \emph{flat path structures},
a path structure on a three-manifold being a pair $(E^\alpha,E^\beta)$
of transverse rank one distributions
in the tangent bundle whose sum $E^\alpha\oplus E^\beta$ is a contact distribution
(we will give in  subsection \ref{subsection-flagstructures} below
more details about this interpretation).
Although any orientable closed three-manifold can be given a path structure,
it is a basic and open problem to decide which
ones bear a flag structure.
For instance, while a flag structure was constructed
by the first author and Thebaldi
on a
(finite-volume and complete) non-compact hyperbolic three-manifold in \cite{falbel},
it is not known whether a flag structure exists or not on a \emph{closed}
hyperbolic three-manifold.

\subsection{A rank-two surgery}\label{subsection-ranktwo}
\par In the rank-one flat conformal or CR-spherical geometries,
it is the existence of North-South dynamics for
the action of any generic element
which ultimately allows one to endow the connected sum of two
geometric manifolds
with a compatible structure.
On the other hand, in the dynamics of the rank two group $\PGL{3}$ on the
flag space $\X$,
the  attracting and repelling points
of North-South dynamics are replaced by geometric one-dimensional objects
that we call \emph{$\alpha-\beta$ bouquet of circles},
described as follows.
\par The flag space $\X$ enjoys two natural
$\PGL{3}$-equivariant circle-bundle projections
$\pi_\alpha$ and $\pi_\beta$,
which are the respective first and second coordinate projections
onto $\RP{2}$ and $\RP{2}_*$, and whose
fibers define two $\PGL{3}$-invariant
transverse foliations of $\X$ by circles, respectively called
\emph{$\alpha$} and \emph{$\beta$-circles} $\Calpha(x)$ and $\Cbeta(x)$.
These foliations being $\PGL{3}$-invariant they define on any flag manifold $M$
two transverse one-dimensional foliations $\Falpha$ and $\Fbeta$, and
through any point $x$ in $\X$ passes thus a 
\emph{$\alpha-\beta$ bouquet}
$\Balphabeta(x)=\Calpha(x)\cup\Cbeta(x)$ of two circles,
which are the attracting sets of loxodromic elements
of $\PGL{3}$ acting on $\X$ as we will
see in subsection \ref{soussection-dynamiqueX}.
\par In order to obtain a surgery we need our manifolds to contain full
$\alpha-\beta$ bouquet of circles.
We define then a
\emph{flag surgery along $\alpha-\beta$ bouquets of circles}
properly introduced  in Definition
\ref{definition-connectedsum}.
The gluing being made along the boundary of a tubular neighbourhood
of the attracting set,
and a neighbourhood of a bouquet of two circles being a genus two handlebody,
the flag surgery is a geometric realization of the gluing
of two three-manifolds along two genus two handlebodies
(see  subsection \ref{subsection-definitionexemples}
for more details).
 Note that a similar phenomenon already appeared
in \cite{francesLorentzianKleinianGroups2004}
where the case of conformal Lorentzian geometry was investigated
-- another rank two geometry.
In this case, the geometric surgeries were however not defined
by gluing along genus two handlebodies, but along
solid tori.
The following result providing flag surgeries
is proved in
section \ref{subsection-preuvetheoremconnectedsum}.
\begin{theoremintro}\label{theoremintro-sommeconnexe}
Let $M$ and $N$ be two flag manifolds,
and $B_M\subset M$, $B_N\subset N$ be two $\alpha-\beta$ bouquets of circles
admitting open neighbourhoods flag isomorphic to open subsets of $\X$.
There exists then a flag surgery of $M$ and $N$ along $B_M$ and $B_N$.
\end{theoremintro}
 It was established in \cite{kulkarni_principle_1978},
that there exits a conformal structure 
 on a connected sum
of conformal structures.   This fact was systematically used
in \cite{kulkarni_uniformization_1986}
to construct examples
enjoying different geometrical properties.
The authors show for instance
that connected sums
of Kleinian Möbius structures are Kleinian,
but they also use these geometric connected sums
to obtain examples of non-Kleinian structures.
Inspired by \cite{kulkarni_uniformization_1986},
one of our main motivations for
Theorem \ref{theoremintro-sommeconnexe}
is to provide new examples
of flag structures
both Kleinian and non-Kleinian.

\subsection{Combination of Kleinian flag manifolds by surgery}\label{subsection-mainresults}
\par For any discrete subgroup $\Gamma\subset\PGL{3}$
acting properly discontinuously on an
open subset $\Omega\subset\X$,
the quotient $\Gamma\backslash\Omega$ bears
a canonical flag structure
and such flag structures are called \emph{Kleinian}.
Kleinian examples will be produced by surgery through
the following result
proved in subsection \ref{subsubsection-Kleinianexamples}
(see Theorem \ref{theorem-Kleinian-connectedsum}),
following an argument initially proved in
\cite[\S 5.6]{kulkarni_uniformization_1986} for conformal structures.
\begin{theoremintro}\label{theoremintro-Kleinianconnectedsum}
A flag surgery of  Kleinian flag manifolds is
a  Kleinian flag structure.
 \end{theoremintro}

Actually,
most of the known Kleinian flag structures
arise from \emph{Anosov representations},
forming an important and deeply studied class of examples.
Such representations $\rho\colon\Gamma\to\PGL{3}$
of hyperbolic groups $\Gamma$ into $\PGL{3}$
admit a $\rho(\Gamma)$-invariant
open subset $\Omega_\rho\subset\X$ with a proper,
discontinuous and cocompact
action of $\rho(\Gamma)$,
and yield thus a closed
Kleinian flag manifold $\rho(\Gamma)\backslash\Omega_\rho$.
We refer to \cite{barbot_flag_2001,barbot_three-dimensional_2010} for the first examples
of flag Anosov representations (having actually appeared before the
introduction of Anosov representations themselves)
and to \cite{guichard_anosov_2012,kapovich_dynamics_2018} for the general theory,
providing domains of discontinuity for Anosov subgroups
in a more general setting.
 In the recent preprint
\cite{nolteConcaveFoliatedFlag2024}, three-dimensional
flag structures with a Hitchin holonomy
are studied in link with some specific distinguished
foliations.

A rich class of  examples  are provided
by \emph{Schottky flag manifolds},
obtained from free Anosov subgroups of $\PGL{3}$
which are Schottky in a natural sense
defined in \cite[\S 1.3.1]{mion-mouton_geometrical_2022}.
Building up on the latter,
Theorem \ref{theoremintro-Kleinianconnectedsum}
yields the following examples of Kleinian flag manifolds
(see Corollary \ref{corollary-newKleinianexamples}).
\begin{corollaryintro}\label{corollaryintro-newKleinian}
 Let $M$ and $N$ be two Schottky flag manifolds,
 and $B_M$, $B_N$ be two $\alpha-\beta$ bouquets of circles
 contained in the respective fundamental domains of $M$ and $N$.
There exists then a flag surgery of $M$ and $N$
 along $B_M$ and $B_N$,
 which is a Kleinian flag manifold.
\end{corollaryintro}

\par Results generalizing the Klein-Maskit combination theorem
were proved for Anosov subgroups in \cite{dey_combination_2019,dey_klein-maskit_2023},
to which Theorem \ref{theoremintro-Kleinianconnectedsum}
gives a concrete geometric interpretation
in the case where the considered Kleinian flag manifolds
are quotients of Anosov subgroups.
However, not all Kleinian flag manifolds arise from Anosov subgroups,
and  Theorem \ref{theoremintro-Kleinianconnectedsum} gives thus
a more general ``combination'' result for Kleinian flag manifolds.
A first necessary condition for a Kleinian example to arise from an Anosov subgroup
is indeed that the holonomy group
should be word-hyperbolic,
which excludes, for instance, the following examples.
\par The action of the Heisenberg group $\Heis{3}\subset\PGL{3}$ on $\X$
admits an open orbit $O$ (and actually only one).
The action of the lattice $\mathrm{Heis}_\Z(3)$ on $O$ is thus properly discontinuous
and cocompact. This yields a Kleinian flag manifold $\mathrm{Heis}_\Z(3)\backslash O$
whose holonomy is nilpotent, preventing it to arise from an Anosov representation
(see \cite[\S 4.2.3]{mion-mouton_partially_2022} for more details on these examples).
We note, however, that the flag surgeries cannot be applied to these flag manifolds,
since they do not contain any $\alpha-\beta$ bouquet of circles.
The $\alpha$ and $\beta$-leaves of these nil-manifolds examples are indeed the stable and unstable
leaves of partially hyperbolic diffeomorphisms
(see \cite[\S 1.1]{mion-mouton_partially_2022}),
and as such none of them is closed.
The authors do not know of any example of Kleinian flag manifold which does not arise
up to finite index
from an Anosov
subgroup and which contains an embedded $\alpha-\beta$ bouquet of circles.

\subsection{Flag structures beyond Kleinian examples}
\label{subsection-nonKleinian}
To the best of our knowledge 
all of the flag structures described in the literature are Kleinian
(most of them arising furthermore from \emph{Anosov representations}
as we explained previously),
to the exception of
the non-compact flag structure constructed in \cite{falbel}
on the complement of a knot.
A general way to produce examples of closed flag structures
(more generally of $(G,X)$-structures) is the
\emph{Ehresman-Thurston} principle,
asserting that the set of morphisms from $\piun{M}$ to $\PGL{3}$
that are holonomy morphisms of a flag structure on a closed manifold $M$
is open.
In a very specific situation,
suggested to the authors by Charles Frances and described in subsection
\ref{subsection-deformationstriviales},
this approach yields indeed non-Kleinian deformations.
These are however topologically constrained to $\Sigma_2\times\Sn{1}$
with $\Sigma_2$ a genus two closed, connected and orientable surface.
\par The initial motivation
of this paper
and of the surgery introduced therein
was precisely to widen the realm of known flag structures
by producing other \emph{non-Kleinian} closed flag manifolds.
In the following result
proved in subsections
\ref{soussection-suitablecovering}
and \ref{subsection-proofnonKleiniangluings},
we use
flag structures surgeries
to provide a general recipe
to construct non-Kleinian flag manifolds.
We denote by $\Sigma_3$ the genus three closed, orientable and connected surface,
and we say that a flag structure is \emph{virtually Kleinian}
if it has a Kleinian covering.

\begin{theoremintro}\label{theoremintro-structuresnonKleiniennesparsommeconnexe}
There exists a virtually Kleinian flag structure on $M=\Sigma_3\times\Sn{1}$
such that,
for any $\alpha-\beta$ bouquet $B_M\subset M$ of two circles,
and for any Kleinian flag manifold $N$ satisfying
 \begin{enumerate}
 \item $N$ contains an $\alpha-\beta$ bouquet
 $B_N$ of two circles admitting a neighbourhood which embeds in $\X$,
\item the holonomy group of $N$ contains a loxodromic element,
\end{enumerate}
we have the following. For any flag surgery $S$ of $M$ and $N$ along $B_M$ and $B_N$, the
 developing map $\delta$ of $S$ is surjective onto $\X$.
 In particular $\delta$ is not a covering map,
 and $S$ is not virtually Kleinian.
\end{theoremintro}

Proposition \ref{proposition-bouquetsSchottky}
shows that Theorem \ref{theoremintro-structuresnonKleiniennesparsommeconnexe}
can may be applied when $N$ is a Schottky flag manifold.
This construction is inspired by the one of \cite{kulkarni_uniformization_1986}
for conformal geometry and of \cite{falbel_spherical_1994} for CR structures.

\subsection{Further questions}
The flag surgery introduced here raises a number of questions
that the authors hope to consider in the future.
The first question concerns the possible \emph{deformations}
of flag surgeries for fixed initial flag manifolds
and $\alpha-\beta$ bouquets inside those.
This question was investigated in \cite{izeki_deformation_1996}
for conformal structures, where non-trivial deformations in the Teichmüller space
of conformal structures were described.
\par Another question is the higher-dimensional generalization
of the procedure described here.
There are several  higher-dimensional analogs to flag structures,
among which there are those modelled  on complete flags but also structures modelled on partial flags. The spaces
$\mathbf{X}_{2n+1}$
of pointed projective hyperplanes of $\RP{n+1}$
under the action of $\PGL{n+2}$
are particularly interesting to the authors.
These are indeed the flat \emph{Lagrangian-contact structures},
which are the natural higher-dimensional analog to path structures.  

\subsection{Acknowledgments}
The authors would like to thank Charles Frances
for his interest and suggestions about this work,
and especially for initially suggesting them to consider the
deformations described in subsection
\ref{subsection-deformationstriviales}.
The second author would like to thank Jean-Philippe Burelle
for an interesting and motivating discussion about deformations of flag structures.
Lastly, the authors thank the anonymous referees for their multiple
helpful and careful remarks and suggestions.

\section{Flag structures and Schottky flag manifolds}
In this section, we summarize the geometric and dynamical  information
that we will need about the flag space and the action of $\PGL{3}$,
and present the examples described in
\cite{mion-mouton_geometrical_2022}
on which our constructions will be based.
We refer to the latter paper for more details
and for the proofs of the results claimed in this section.

\subsection{Flag structures}\label{subsection-flagstructures}
The flag space
$\X$ is an orientable
three-dimensional closed manifold,
with universal cover $\Sn{3}$ and fundamental group of  cardinality 8.
The stabilizer of the standard flag
$o\coloneqq([e_1],[e_1,e_2])$ for the transitive action of $\PGL{3}$ is the subgroup
$\Pmin\subset\PGL{3}$ of upper-triangular matrices,
minimal parabolic subgroup of $\PGL{3}$
(where $(e_i)$ denotes in all of this paper the standard basis of $\R^n$).
The action of $\PGL{3}$ induces thus an equivariant identification of $\X$
with the homogeneous space $\PGL{3}/\Pmin$.
\par A structure locally modelled
on $\X$ is formalized in the following way.
\begin{definition}\label{definition-flagstructure}
A \emph{$(\PGL{3},\X)$-atlas}
on a three-manifold $M$
is an atlas of connected charts of $M$ with values in $\X$, whose transition functions
are restrictions of elements of $\PGL{3}$.
A \emph{$(\PGL{3},\X)$-structure},
or \emph{flag structure} on $M$
is
a maximal $(\PGL{3},\X)$-atlas on $M$,
and a \emph{flag manifold} a three-manifold endowed with a flag structure.
 A \emph{$(\PGL{3},\X)$-morphism}, or \emph{flag morphism}
 between two flag structures is a map
 which reads in any connected
 $(\PGL{3},\X)$-charts as the restriction of an element
 of $\PGL{3}$.
\end{definition}
We recall (see for instance
\cite{thurston_three-dimensional_1997,canary_notes_1987} for more details)
that for any $(\PGL{3},\X)$-structure on $M$,
there exists:
\begin{enumerate}
 \item a local diffeomorphism
 $\delta\colon\tilde{M}\to\X$ which is a
 $(\PGL{3},\X)$-morphism, called the \emph{developing map},
 \item and a \emph{holonomy morphism} $\rho\colon\piun{M}\to\PGL{3}$
 for which $\delta$ is $\rho$-equivariant.
\end{enumerate}
Moreover,
if the  flag structure is fixed, then
such a pair $(\delta,\rho)$ is unique up to the action
$g\cdot(\delta,\rho)=(g\circ\delta,g\rho g^{-1})$
of $\PGL{3}$
and, reciprocally, any such pair defines a unique compatible flag structure.
\begin{definition}\label{definition-alphabetacircles}
 The $\alpha$- and $\beta$-circles of
$x=(p,D)\in\X$ are denoted by
\begin{equation}\label{equation-cerclesalphabeta}
 \Calpha(x)=\enstq{(p,D')}{D'\ni p}=\Calpha(p)
 \text{~and~}
 \Cbeta(x)=\enstq{(p',D)}{p'\in D}=\Cbeta(D).
\end{equation}
The $\alpha$- and $\beta$-circles of $\X$, being $\PGL{3}$-equivariant,
induce on any flag manifold $M$ a pair
$\Lm_M=(\Falpha,\Fbeta)$
of transverse one-dimensional $\alpha$- and $\beta$-foliations,
which we call the \emph{flat path structure}
induced by the flag structure of $M$.
\end{definition}
\par The reason for this  terminology is that
the pair of $\alpha$- and $\beta$-foliations
induced on a three-manifold by a flag structure
are specific instances of a \emph{path structure},
which is a pair $(E^\alpha,E^\beta)$
of smooth transverse line fields
on a three-manifold whose sum $E^\alpha\oplus E^\beta$ is a contact structure
(see \cite{ivey_cartan_2016}
or \cite[\S 1.2]{mion-mouton_geometrical_2022} for more details).
A path structure $(E^\alpha,E^\beta)$ is indeed equivalent
to its integral foliations $(\Falpha,\Fbeta)$,
and it can be checked that the $\alpha$- and $\beta$-foliations
of a flag structure define a path structure,
\emph{i.e.} that $\Fitan{\Calpha}\oplus\Fitan{\Cbeta}$
is a contact distribution in $\X$.
Path structures can be interpreted as \emph{Cartan geometries} modelled
on the flag space $\X$
which endows them with a notion of \emph{curvature},
and the path structures induced by flag structures are precisely the ones whose
curvature vanishes and are therefore called \emph{flat}
(see \cite{sharpe_differential_1997} or
\cite[\S 2]{mion-mouton_partially_2022} for more details).
A flag structure and its associated flat path structure
turns out to be equivalent
in the following sense.
\begin{proposition}\label{proposition-Liouville}
Let $M$ and $N$ be two flag manifolds,
of associated flat path structures $\Lm_M$ and $\Lm_N$.
Then a diffeomorphism $f\colon M\to N$
is a flag isomorphism if, and only if it is a path structure isomorphism.
\end{proposition}
A diffeomorphism $f\colon M\to N$ is a path structure isomorphism if
for any $x\in M$:
$\Diff{x}{f}(E^\alpha_M(x))=E^\alpha_N(f(x))$ and
$\Diff{x}{f}(E^\beta_M(x))=E^\beta_N(f(x))$.
\begin{proof}[Proof of Proposition \ref{proposition-Liouville}]
 The direct implication is straightforward,
 and the reverse one follows from a path structure analog of
 the ``Liouville theorem''.
 Let $f\colon M\to N$ be a path structure isomorphism from $\Lm_M$ to $\Lm_N$.
 Then in any two $(\PGL{3},\X)$-charts
 $\varphi$ and $\psi$ from connected open subsets $U\subset M$ and $f(U)\subset N$,
 $F=\psi\circ f\circ\varphi^{-1}$ is a diffeomorphism between the connected open subsets
 $\varphi(U)$ and $\psi\circ f(U)$ of $\X$,
 preserving the flat path structure of $\X$ defined by $\alpha$ and $\beta$-circles.
 According to the ``Liouville theorem''
 \cite[Theorem 2.9]{mion-mouton_partially_2022},
 $F$ is thus the restriction of an element $g\in\PGL{3}$.
 Hence $f$ is indeed a $(\PGL{3},\X)$-isomorphism which concludes the proof.
\end{proof}

\subsection{Dynamics of $\PGL{3}$ on $\X$}\label{soussection-dynamiqueX}
 It is known that the action of
the conformal group $\PBio{1}{n+1}$ on the round sphere $\Sn{n}$
exhibits a dynamics
called ``North-South''.
For any sequence $(g_n)$ going to infinity,
there exists a subsequence of
$(g_n)$ which has a repelling point $p_-$
and an attracting point $p_+$,
and the restriction of $(g_n)$ to $\Sn{n}\setminus\{p_-\}$
converges moreover to the constant map $p_+$.
This dynamics is due to the fact that
$\PBio{1}{n+1}$ has rank one
as opposed to $\PGL{3}$ which has \emph{rank two},
meaning that the Cartan subalgebra
$\mathfrak{a}$ of $\slR{3}$ has dimension two.
The dimension two leaves space
in a Weyl chamber of $\mathfrak{a}$
for three different ways to go
to infinity,
namely along one of the walls,
or in the interior of the chamber.
In the latter case however,
a kind of North-South dynamics persists
with the attracting and repelling points
being respectively replaced by an attracting and a repelling
\emph{$\alpha-\beta$ bouquet of circles}.
\begin{definition}\label{definition-alphabetabouquet}
 For any $x\in\X$, the \emph{$\alpha-\beta$ bouquet} of $x$
 is the union
 \begin{equation*}\label{equation-alphabetabouquet}
 \Balphabeta(x)=\Calpha(x)\cup\Cbeta(x)
 \end{equation*}
 of the $\alpha$
 and $\beta$-circles of $x$.
\end{definition}
\par Let us say that $g\in\PGL{3}$
is \emph{loxodromic} if any representative of $g$
has three real eigenvalues
of pairwise distinct absolute values $a>b>c>0$,
whose corresponding eigenlines are denoted by $p_+$, $p_\pm$ and $p_-$.
Then $x_-=(p_-,[p_-,p_\pm])$ and $x_+=(p_+,[p_+,p_\pm])$
are the \emph{ repelling and  attracting fixed points of $g$} in $\X$,
and we call their $\alpha-\beta$ bouquet of circles
\begin{equation}\label{equation-BmoinsBplusdeg}
 \Bmoins(g)=\Calpha(x_-)\cup\Cbeta(x_-)
 \text{~and~}
 \Bplus(g)=\Calpha(x_+)\cup\Cbeta(x_+)
\end{equation}
the \emph{ repelling and  attracting bouquet of circles of $g$}.
Note in particular that $\Bplus(g^{-1})=\Bmoins(g)$
(see for instance \cite[Remark 2.14]{mion-mouton_geometrical_2022}).
\begin{example}\label{example-matricediagonalebouquets}
 Let $g=\Diag(a,b,c)$ be a diagonal matrix
with $a>b>c>0$.
Then
$x_-=([e_1],[e_1,e_2])$, $x_+=([e_3],[e_2,e_3])$, and thus:
\begin{equation*}\label{equation-bouquetsdiagonal}
 \Bmoins(g)=\Calpha[e_1]\cup\Cbeta[e_1,e_2]
 \text{~and~}
 \Bplus(g)=\Calpha[e_3]\cup\Cbeta[e_2,e_3].
\end{equation*}
\end{example}

The set of compact subsets of $\X$ is endowed with the
topology defined by the \emph{Hausdorff distance}
(induced by any Riemannian metric on $\X$,
see for instance \cite[\S 2.1]{mion-mouton_geometrical_2022}
for more details),
for which it is a compact metrizable space.
We have then the following result,
which is a direct reformulation of
\cite[Lemma 2.2, Lemma 2.21
and Example 2.22]{mion-mouton_geometrical_2022}.
\begin{lemma}\label{lemma-bouquetattractifrepulsif}
Let $g\in\PGL{3}$ be a loxodromic element and
$K\subset\X\setminus\Bmoins(g)$ be a compact subset.
Then any accumulation point
of the sequence $(g^n(K))$
is contained in $\Bplus(g)$.
\end{lemma}

\subsection{Schottky flag manifolds}\label{soussection-Schottkyungenerateur}
Let $\{g^t\}_{t\in\R}\subset\PGL{3}$
be a 1-parameter subgroup of loxodromic elements,
of respective repelling and attracting points $x_-,x_+\in\X$,
and repelling and attracting bouquet of circles
$B^-=\Balphabeta(x_-)$, $B^+=\Balphabeta(x_+)$.
Let $\Omega=\X\setminus(B^-\cup B^+)$,
$g=g^1$ and
$\Gamma$ be the cyclic group $\langle g \rangle\simeq\langle\Z\rangle$.
Then according to \cite[Lemma 3.1]{mion-mouton_geometrical_2022},
there exists a compact connected neighbourhood $H$ of $B^-$,
which is  homeomorphic to a genus two handlebody and such that:
\begin{enumerate}
 \item $\Sigma\coloneqq\partial H$ is a genus  two closed surface,
 transverse to the orbits of $\{g^t\}$;
 \item and
\begin{equation}\label{equation-parametrisationOmega}
 (x,t)\in\Sigma\times\R\mapsto g^t(x)\in\Omega
\end{equation}
is a diffeomorphism.
\end{enumerate}
 We recall that a genus two handlebody is
a compact, orientable and connected three-manifold with boundary,
which is obtained from a three-ball by attaching
two 1-handles.
Equivalently, it is obtained from a solid torus by attaching a single
1-handle.
\par The diffeomorphism \eqref{equation-parametrisationOmega}
induces a diffeomorphism from
$\Sigma\times\Sn{1}$ to the quotient
$M_0\coloneqq\Gamma\backslash\Omega$
(where $\Sn{1}=\R/\Z$),
as well as an identification between the fundamental group of $M_0$
and $\piun{\Sigma}\times\Z$
($\piun{\Omega}$ being identified with $\piun{\Sigma}$
since $\Sigma$ is a deformation retract of $\Omega$).
Furthermore with $H^-\coloneqq H$,
$H^+\coloneqq \X\setminus \Int(g(H^-))$ and 
\begin{equation}\label{equation-definitionDdomainefondamental}
 D\coloneqq \X\setminus(H^-\cup H^+),
\end{equation}
$D$ is a fundamental domain for the action of $\Gamma$ on $\Omega$,
and is naturally identified
to $\Sigma\times (0,1)$
by the diffeomorphism \eqref{equation-parametrisationOmega}.
\par Since $M_0$ is the quotient of $\Omega$ by the free
and proper action of $\Gamma$, it inherits a Kleinian flag structure
from the one of $\Omega\subset\X$,
and we now describe the developing map and holonomy morphism
of this flag structure.
With $\pi_\Omega\colon E\to\Omega$ the universal covering map
of $\Omega$
and $\pi_{\Gamma}\colon\Omega\to M_0$ the canonical projection,
$\pi_{\Gamma}\circ\pi_{\Omega}\colon E\to M_0$
is the universal covering map of $M_0$.
Hence $\piun{M_0}\equiv\piun{\Sigma}\times\Z$
acts on $E$, $\pi_\Omega$
is $\piun{\Sigma}\equiv\piun{\Omega}$-invariant,
and is thus equivariant for the action of
$\piun{M_0}\equiv\piun{\Sigma}\times\Z$
with respect to the morphism
\begin{equation}\label{equation-rhoM0}
 \rho_{M_0}\colon(\lambda,n)\in\piun{\Sigma}\times\Z\mapsto g^n.
\end{equation}
Hence $(\pi_\Omega,\rho_{M_0})$
is a pair of developing map and holonomy morphism of $M_0$,
and we denote $\delta_{M_0}=\pi_\Omega$.
\par In \cite[\S 1.3.1 Theorem D]{mion-mouton_geometrical_2022},
the second author describes
Kleinian flag manifolds $\Gamma\backslash\Omega(\Gamma)$ obtained from
\emph{Schottky subgroups} of $\PGL{3}$, that we will call
\emph{Schottky flag manifolds}.
These examples are the analog of the previous
construction
for non-cyclic free subgroups of $\PGL{3}$
(their existence also follows from \cite{guichard_anosov_2012}).

\subsection{An obstruction to be Kleinian}
 An important and basic question
concerning locally homogeneous structures
is to find examples of \emph{non-Kleinian} structures, and
to this end it is necessary to find an obstruction
for a structure to be Kleinian.
In this subsection we establish such a useful criterion
for flag structures,
relying on dynamical information on the holonomy group
of the structure.
This is inspired from the criterion initially developed in
\cite[p.20-22]{kulkarni_uniformization_1986} for
conformal Riemannian structures.

\begin{lemma}\label{lemma-KulkarniPinkall}
Let $\delta(\tilde{M})\subset\X$
 be the developing map image
 and
 $\Gamma\subset\PGL{3}$ be the
 holonomy group of a flag structure on a closed three-dimensional
 manifold $M$.
\begin{enumerate}
 \item Assume that there exists a loxodromic element $g\in\Gamma$
 such that $\Bmoins(g)\subset\delta(\tilde{M})$.
Then $\X\setminus\Bplus(g)\subset\delta(\tilde{M})$.
\item Assume that $\Gamma$ contains a loxodromic element
$g$ such that
$\delta(\tilde{M})$ contains the attracting and repelling bouquets
$\Bmoins(g)$ and $\Bplus(g)$ of $g$.
Then $\delta(\tilde{M})=\X$.
\item If $\delta(\tilde{M})=\X$ and $\piun{M}$ is infinite,
then $\delta$ is not a covering map.
In particular, $M$ is not virtually Kleinian.
\end{enumerate}
\end{lemma}
\begin{proof}
1. Since $\Bmoins(g)\subset\delta(\tilde{M})$,
there exists a compact neighbourhood $P$ of $\Bmoins(g)$
which is contained in $\delta(\tilde{M})$.
Then with $K=\X\setminus\Int(P)$,
any accumulation point of $g^n(K)$ is contained
in $\Bplus(g)$ according to Lemma \ref{lemma-bouquetattractifrepulsif}.
In particular for any open neighbourhood $O$ of $\Bplus(g)$,
there exists $n$ such that $g^n(K)\subset O$,
\emph{i.e.} such that $\X\setminus O\subset g^n(\Int(P))$.
But $\delta(\tilde{M})$ is $\Gamma$-invariant,
hence $g^n(\Int(P))\subset\delta(\tilde{M})$ and thus
$\X\setminus O\subset\delta(\tilde{M})$.
Finally
$\X\setminus\Bplus(g)=
\cup_{O}\X\setminus O
\subset\delta(\tilde{M})$,
the union being taken on neighbourhoods $O$ of $\Bplus(g)$,
which concludes the proof of the claim. \\
2. According to the first claim $\delta(\tilde{M})$
contains $\X\setminus\Bplus(g)$,
and it contains $\Bplus(g)$ by assumption,
hence $\delta(\tilde{M})=\X$. \\
3. Assume by contradiction that $\delta$ is a covering map onto $\X$.
Then since $\piun{\X}$ is finite,
$\delta$ is a finite degree covering
and $\tilde{M}$ is thus compact.
This contradicts the fact that $\piun{M}$ is infinite,
and $\delta$ is thus not a covering map.
Assume now by contradiction that $M'$ is a covering of $M$ which is Kleinian.
Then the developing map of $M'$ remains $\delta\colon\tilde{M}\to\X$,
and is a covering map since $M'$ is Kleinian. This contradicts
what we have proved and concludes the proof.
\end{proof}

\section{Flag surgeries and Kleinian examples}\label{section-connectedsum}

\subsection{Definition and main result}\label{subsection-definitionexemples}
\par Let $M,N$ be two
three-manifolds,
$H_M\subset M$ and $H_N\subset N$
be
compact submanifolds of dimension three with boundary  of $M$ and $N$,
and $f\colon \partial H_M\to \partial H_N$ be a
diffeomorphism
between their boundaries
(which are closed two-dimensional
submanifolds in $M$).
Then the \emph{surgery of $M$ and $N$ along $f$}
\begin{equation}\label{equation-definitionconnectedsum}
M\#_f N\coloneqq
 (M\setminus \Int(H_M))\sqcup(N\setminus\Int(H_N))/
 \{\forall x\in \partial H_M :x\sim f(x)\}
\end{equation}
contains natural embeddings of
$M\setminus \Int(H_M)$ and $N\setminus \Int(H_N)$
denoted by $j_M$ and $j_N$,
and has a unique natural smooth structure
extending the ones of
$j_M(M\setminus \Int(H_M))$ and $j_N(N\setminus \Int(H_N))$.
Moreover,
the homeomorphism type of $M\#_f N$  only depends on the homotopy type
of $f$.  For instance, if $H_M$ and $H_N$ are balls then $M\#_f N=M\# N$
is simply the connected sum of $M$ and $N$.
\par If $M$ and $N$ are two flat conformal Riemannian manifolds, then
it is possible to endow their connected sum $M\# N$ with a compatible
flat conformal Riemannian structure, as shown in \cite{kulkarni_principle_1978}.
This construction relies on the existence of specific conformal
automorphisms:
for any open set $U\subset\Sn{n}$,
there exists an \emph{inversion} reversing the two
boundary spheres of $B_2\setminus B_1$,
with $B_1\subset B_2$ two balls contained in $U$.
 An important feature of the construction
of the geometric connected sum in \cite{kulkarni_principle_1978}
is that the spheres $\partial B_1$ and $\partial B_2$
which are exchanged by the inversion,
can be chosen as small as one wants.
\par For flag structures however, such geometric inversions
with respect to spheres
do not exist anymore.
Indeed,
because of the rank two dynamics
described in subsection \ref{soussection-dynamiqueX},
the best that we can hope for is  an involution exchanging two
$\alpha-\beta$ bouquet of circles.
These involutions will exchange the boundaries of two nested
neighbourhoods of a bouquet of two circles,
which are genus two handlebodies.
\par Our goal in this section is to construct a compatible flag structure
on the surgery of two flag manifolds
along genus two handlebodies.
In a flag manifold $M$,
we call \emph{$\alpha-\beta$ bouquet of circles}
a bouquet of two circles
 which is the union
$B=\Falpha(x)\cup\Fbeta(x)$
of two closed $\alpha$- and $\beta$-leaves
intersecting at a single point. For the following definition see Figure \ref{surgery}.

\begin{definition}\label{definition-connectedsum}
 Let $M$ and $N$ be two flag manifolds,
 and $B_M\subset M$, $B_N\subset N$ be two $\alpha-\beta$ bouquet of circles.
 We will say that a flag manifold $S$ is
 \emph{a flag surgery of $M$ and $N$ along $B_M$ and $B_N$},
 if there exists:
 \begin{enumerate}
  \item genus two handlebodies $K_M\subset M$, respectively $K_N\subset N$,
  containing $B_M$, resp. $B_N$ in their interiors,
  and flag structure embeddings $j_M\colon M\setminus K_M\to S$ of image $M'$
  and $j_N\colon N\setminus K_N\to S$ of image $N'$,
  such that $S=M'\cup N'$;
  \item and open subsets $U_M\supset K_M$ and $U_N\supset K_N$
  with flag structure embeddings into $\X$,
  such that $M'\cap N'\subset V=j_M(U_M \setminus K_M)\cap j_N(U_N \setminus K_M)$,
  with $\partial M'$ and $\partial N'$ isotopic within $V$.
 \end{enumerate}
\end{definition}

\begin{figure}
\begin{center}
\begin{overpic}[scale=0.5]{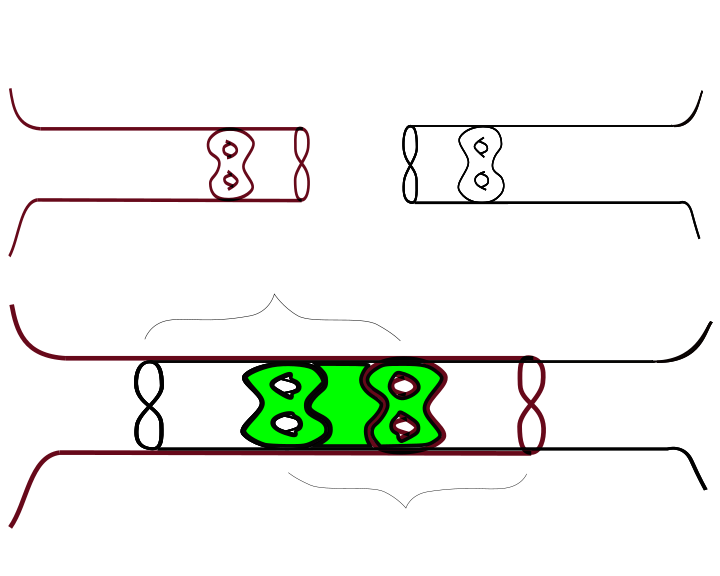}
\put(37,40){$U_N$}
\put(55,5){$U_M$}
\put(65,24){$\small{K_M}$}
\put(25,24){$\small{K_N}$}
\put(3,55){$M$}
\put(92,55){$N$}
\put(43,55){$B_M$}
\put(50,55){$B_N$}
\put(35,55){$\small{K_M}$}
\put(58.5,55){$\small{K_N}$}

\end{overpic}
\end{center}
\caption{The surgery $S$ along $B_M$ and $B_N$. We glue the two manifolds $M$ and $N$ through embeddings $j_M: M\setminus K_M\to S$ and $j_N: N\setminus K_N\to S$.  The regions $K_N$ and $K_M$ are deleted and  the green region  $j_M(U_M\setminus K_M)\cap j_N(U_N\setminus K_N)$ contains the identifications defined by the surgery on a tubular neighborhood of a genus two surface.}\label{surgery}
\end{figure}

This definition directly implies that
a flag surgery $S$ of $M$ and $N$
along $B_M$ and $B_N$
is diffeomorphic to a surgery of $M$ and $N$ along some diffeomorphism
$f\colon \partial H_M\to \partial H_N$,
with $H_M\subset U_M$ a genus two handlebody containing $B_M$ in its interior
and of boundary isotopic to $\partial K_M$ and likewise for $H_N$.
Definition \ref{definition-connectedsum} fulfills thus
our initial goal, namely to geometrically realize
topological surgeries along genus two handlebodies,
its condition (1) translating the compatibility
of the surgery with the structures of the two pieces.

\begin{remark}\label{remark-hypothesebouquets}
If the condition (2) of Definition \ref{definition-connectedsum}
asking for a flag embedding of a neighbourhood of the $\alpha-\beta$ bouquets
is absent in the conformal case,
it is simply because any open subset of $\R^n$ contains
 a conformally embedded euclidean ball,
 and thus so does
 any conformally flat manifold.
 This allows one to form conformal connected sums of any two conformally flat
 manifolds, while it is not clear that any
 $\alpha-\beta$ bouquet of circles in a flag manifold
 admits a neighbourhood embedding in $\X$.
\end{remark}

The necessary condition (2) of Definition
\ref{definition-connectedsum}
imposed on the $\alpha-\beta$ bouquets
is sufficient to
obtain the existence of the surgery,
given by Theorem \ref{theoremintro-sommeconnexe}
in  subsection \ref{subsection-ranktwo}
and proved in subsection \ref{subsection-preuvetheoremconnectedsum} below.

\subsection{Anti-flag involutions of the flag space}
\label{subsection-dualinvolution}
Denoting by $V^\bot$ the orthogonal of a vector subspace $V\subset\R^3$
for the standard euclidean quadratic form,
the \emph{standard involution}
\begin{equation}\label{equation-involutiondualeX}
 \kappa\colon
 (p,D)\in\X\mapsto(D^\bot,p^\bot)\in\X
\end{equation}
of the flag space will be used to glue flag structures,
replacing in this context
the inversion of the round sphere used
for conformal structures.
The standard inversion is easily seen to be equivariant
for the involutive morphism
\begin{equation}\label{equation-equivariancekappa}
  \Theta\colon g\in\PGL{3}\to \transp{g}^{-1}\in\PGL{3},
\end{equation}
 where $\transp{g}$ denotes the transpose
of the matrix $g$.
It also exchanges the $\alpha-$ and $\beta-$ foliations,
namely for any $x\in\X$:
\begin{equation}\label{equation-kappeechangealphabeta}
 \kappa(\Calpha(x))=\Cbeta(\kappa(x))
 \text{~and~}
 \kappa(\Cbeta(x))=\Calpha(\kappa(x)).
\end{equation}
\par Observe in particular that $\kappa$ is not an element
of the automorphism group $\PGL{3}$
of the flat path structure $(\Calpha,\Cbeta)$ of $\X$.
The group of automorphisms
of the \emph{unordered} pair $\{\Calpha,\Cbeta\}$ is actually generated
by $\PGL{3}$ and $\kappa$, and its two connected components are
$\PGL{3}$ and the left coset of \emph{anti-flag morphisms},
\emph{i.e.} diffeomorphisms of $\X$ exchanging
the $\alpha$ and $\beta$-foliations,
which are of the form $\kappa\circ g$ for $g\in\PGL{3}$.
Note moreover that
$\kappa$ is not a canonical involution of $\X$,
its choice being equivalent to the one of an euclidean quadratic form on $\R^3$.
The set $g\circ \kappa\circ g^{-1}$ of conjugates of $\kappa$ by elements
of $\PGL{3}$
is however canonical
(a change of quadratic form being indeed
equivalent to conjugating $\kappa$ by some $g$).
These conjugates are precisely the involutive anti-flag morphisms of $\X$,
that we will call more simply \emph{anti-flag involutions}.

\begin{remark}\label{remark-kappabouquets}
Note that for any $x\in\X$, $\kappa(B_{\alpha\beta}(x))$ is disjoint
from $B_{\alpha\beta}(x)$.
\end{remark}

\begin{remark}\label{remark-involutionsdimensionsup}
 The definition of $\kappa$ can be generalized to the higher dimensional
flag spaces
(see more details in \cite{Falbel-Wang}),
and the procedure of flag surgery
could be applied to higher dimensional structures.
In this paper we restrict ourselves to the three-dimensional case,
and the higher-dimensional \emph{Lagrangian-contact structures} or  \emph{complete flag structures}
will be considered in a future work.
\end{remark}

\begin{lemma}\label{lemma-echangecorpsenanses}
 Let $B\subset\X$
 be an $\alpha-\beta$ bouquet of circles and
 $H_2\subset \X$ be a genus two handlebody which is a
 neighbourhood of $B$.
 Then there exists
 a genus two handlebody $H_1$
 which is a neighbourhood of $B$,
 as close as one wants from $B$,
 and $g\in\PGL{3}$ such that with $\varphi=g^{-1}\circ\kappa\circ g$:
 \begin{enumerate}
  \item $H_1\subset\Int(H_2)$,
  \item $\varphi(H_1)=\X\setminus\Int(H_2)$,
  \item $\varphi(\partial H_1)=\partial H_2$,
  \item $\varphi(H_2\setminus\Int(H_1))=H_2\setminus\Int(H_1)$.
 \end{enumerate}
\end{lemma}

\begin{figure}
\begin{center}
\begin{overpic}[scale=0.3]{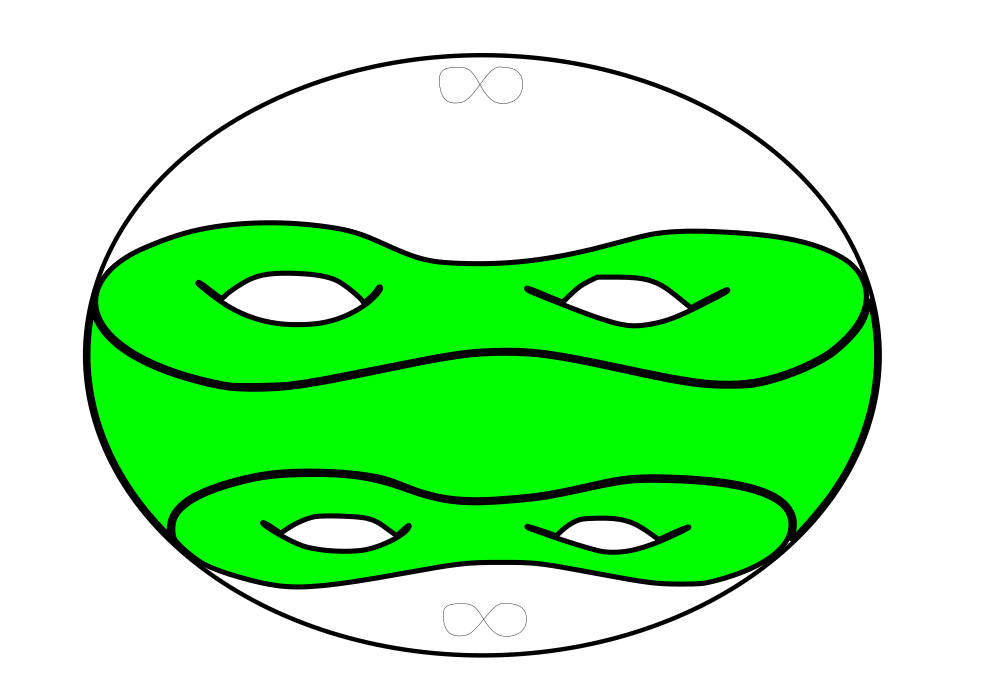}
\put(41,37){$\partial H_2$}
\put(43,13.2){$\partial H_1$}
\put(55,5){$B$}
\put(54,57){$\varphi(B)$}
\put(34,24){$\small{H_2\setminus\Int(H_1)}$}
\put(80,55){$\bf{X}$}
\end{overpic}

\end{center}
\caption{The green region is defined as the intersection of the handlebody $H_2$ and the complement of the handlebody $H_1$. It is  invariant under the map $\varphi=g^{-1}\circ\kappa\circ g$. The $\alpha-\beta$ bouquet is contained in the handlebody $H_1$.}
\label{lemma3.5}
\end{figure}

\begin{proof}
See Figure \ref{lemma3.5}. Let $B=\Calpha(x)\cup\Cbeta(x)$
 and $g$ be a loxodromic element in the stabilizer of $x$
 which has $B$ as repelling bouquet
 (see subsection \ref{soussection-dynamiqueX} for more details).
We will show below that with $\varphi_n=g^{-n}\circ\kappa\circ g^{n}$,
any accumulation point of
$H_1^n\coloneqq\X\setminus\varphi_n^{-1}(\Int(H_2))$
is contained in $B$.
For any neighbourhood $O$ of $B$,
there exists thus $n$ such that $H_1^n\subset O\cap\Int(H_2)$,
and the claims (1), (2) and (3) follow then directly for
$H_1=H_1^n$ and $\varphi=\varphi_n$.
Moreover $\Int(H_1)=\X\setminus\varphi^{-1}(H_2)$,
hence
$\varphi(H_2\setminus\Int(H_1))
=\varphi(H_2\cap\varphi^{-1}(H_2))
=\varphi(H_2)\cap H_2
=H_2\setminus\Int(H_1)$ since $\varphi$ is involutive.
Note that $H_1$ can be chosen in $O$, \emph{i.e.} as close
from $B$ as one wants. \\
For $n_k\to\infty$ such that
$\X\setminus g^{-n_k}\circ\kappa\circ g^{n_k}(\Int(H_2))$
converges to $K_\infty$,
it only remains to show that $K_\infty\subset B$.
Since the compact subset $\X\setminus\Int(H_2)$ is disjoint from
$B=\Bmoins(g)$,
any accumulation point of $g^{n}(\X\setminus\Int(H_2))$
is contained in $\Bplus(g)$ according
to Lemma \ref{lemma-bouquetattractifrepulsif},
\emph{i.e.} any accumulation point of
$\X\setminus\kappa\circ g^n(\Int(H_2))$ is contained in
$\kappa(\Bplus(g))$, which is disjoint from
$\Bplus(g)=\Bmoins(g^{-1})$.
Taking a subsequence, we can thus assume
that $\X\setminus\kappa\circ g^{n_k}(\Int(H_2))\subset P$
with $P\subset\X\setminus\Bmoins(g^{-1})$ a compact subset.
But according to Lemma \ref{lemma-bouquetattractifrepulsif} again,
any accumulation point of $g^{-n}(P)$ is then contained
in $\Bplus(g^{-1})=\Bmoins(g)=B$,
hence $K_\infty=\lim
g^{-n_k}(\X\setminus\kappa\circ g^{n_k}(\Int(H_2)))$
is contained in $B$, which concludes the proof of the Lemma.
\end{proof}

\subsection{Proof of Theorem \ref{theoremintro-sommeconnexe}}
\label{subsection-preuvetheoremconnectedsum}
Possibly taking smaller neighbourhoods
and composing with the action of some element of $\PGL{3}$,
there exists an $\alpha-\beta$ bouquet of circles $B=\Calpha(x)\cup\Cbeta(x)\subset\X$,
connected neighbourhoods $U\subset\X$ of $B$,
$U_M\subset M$ of $B_M$, $U_N\subset N$ of $B_N$,
and flag isomorphisms
$\phi_M\colon U_M\to U$
and $\phi_N\colon U_N\to\kappa(U)$.
Hence $\phi_N'\coloneqq\kappa\circ\phi_N$ is a diffeomorphism from $U_N$ to $U$.
According to Lemma \ref{lemma-echangecorpsenanses},
there exists two genus two handlebodies $H_1$ and $H_2\subset U$
which are neighbourhoods of $B$ satisfying
 $H_1\subset \Int(H_2)$, and $g\in\PGL{3}$
such that $\varphi=g^{-1}\circ\kappa\circ g$ preserves $H_2\setminus\Int(H_1)$
and exchanges $\partial H_1$ and $\partial H_2$.
We now introduce the manifold
\begin{equation}\label{equation-collageH1H2}
 S\coloneqq
 (M\setminus\phi_M^{-1}(H_1))\sqcup
 (N\setminus\phi_N'^{-1}(H_1))/
 \{\forall x\in\Int(H_2)\setminus H_1:
 \phi_M^{-1}(x)\sim\phi_N'^{-1}(\varphi(x))\},
\end{equation}
together with the natural embeddings $j_M$ and $j_N$
of $M\setminus\phi_M^{-1}(H_1)$
and $N\setminus\phi_N'^{-1}(H_1)$ in $S$.
According to the equivariance of $\kappa$
(see \eqref{equation-equivariancekappa}),
$\phi_N'^{-1}\circ\varphi\circ\phi_M$
equals $\phi_N^{-1}\circ \transp{g}g\circ\phi_M$
and is thus a flag morphism
in restriction to $\phi_M^{-1}(\Int(H_2)\setminus H_1)$.
In other words, the equivalence relation $\sim$ defining $S$
 preserves the flag structures of the open subsets
$\phi_M^{-1}(\Int(H_2)\setminus H_1)$
and $\phi_N'^{-1}(\Int(H_2)\setminus H_1)$
of $M$ and $N$.
\par Therefore, the union of the
$(\PGL{3},\X)$-atlases of
$j_M(M\setminus\phi_M^{-1}(H_1))$
and $j_N(N\setminus\phi_N^{-1}(H_1))$
defines a $(\PGL{3},\X)$-atlas on $S$,
which induces by definition the
\emph{canonical flag structure of $S$}.
Conversely, any flag structure on $S$
for which $j_M$ and $j_N$ are $(\PGL{3},\X)$-morphisms
has to coincide with this specific flag structure.
$S$ is a surgery of $M$ and $N$ along $B_M$ and $B_N$,
which concludes the proof of Theorem \ref{theoremintro-sommeconnexe}.
\begin{remark}\label{remark-choix}
 Let us emphasize that in the procedure described previously,
the choice of open subsets used to form the surgery along a given pair of $\alpha-\beta$
bouquets is non-unique.
The investigation of these distinct flag structures will be the subject of a subsequent work.
\end{remark}

\subsection{Kleinian flag manifolds by surgery}\label{subsubsection-Kleinianexamples}
Using the surgery procedure introduced in the previous  subsection
we can prove,
using the same argument as in
\cite[\S 5.6]{kulkarni_uniformization_1986}, the
following result
(Theorem \ref{theoremintro-Kleinianconnectedsum}
of the introduction)
yielding Kleinian flag structures.
We point out a related work in
\cite{dey_combination_2019,dey_klein-maskit_2023}
where combination results generalizing the Klein-Maskit theorem
were proved for Anosov subgroups,
and to which Theorem \ref{theorem-Kleinian-connectedsum} below
gives a concrete geometric interpretation when the considered Kleinian flag manifolds
are quotients of Anosov subgroups.

\begin{theorem}\label{theorem-Kleinian-connectedsum}
 A flag surgery of Kleinian flag manifolds is
 a Kleinian flag structure.
\end{theorem}
\begin{proof}\label{label-proofKleiniangluing}
Let $\Gamma_1$ and $\Gamma_2$ be the holonomy groups
of two Kleinian flag manifolds $M_1=\Gamma_1 \backslash\Omega_1$ and
$M_2=\Gamma_2 \backslash\Omega_2$ (which we can assume to be connected without loss of generality),
$\Omega_1$ and $\Omega_2$ being connected open subsets of $\X$
where $\Gamma_1$ and $\Gamma_2$ act freely and properly discontinuously.
We will say that
an open connected set $D\subset\Omega$ is
a \emph{fundamental domain} for the action of $\Gamma$
on $\Omega$ if it coincides with the interior of its closure,
is disjoint from its translates by any non-trivial element of $\Gamma$,
 $\Omega\subset \Gamma\cdot \bar{D}$
and for any compact subset $K\subset\Omega$:
$\enstq{\gamma\in\Gamma}{\gamma \bar{D}\cap K\neq\varnothing}$ is finite.
The actions of each $\Gamma_i$ on $\Omega_i$ admit a fundamental domain $D_i$ (see, for instance,  \cite[Theorem 25]{kapovich_note_2023}).
\par For $i=1,2$,
let $U_i\subset\Omega_i$ be the neighbourhood of an $\alpha-\beta$ bouquet $B_i$
which embedds in $M_i$,
\emph{i.e.} such that $\overline{U_{i}}\subset D_i$,
and so that $S$ is the flag surgery along $U_1$ and $U_2$
in the following sense.
Using notation as in Definition \ref{definition-connectedsum},
we assume that $\partial j_{M_1}(\mathcal{U}_1)=\partial j_{M_2}(\mathcal{U}_2)$
and that $S=j_{M_1}(M_1\setminus\mathcal{U}_1)\cup j_{M_2}(M_2\setminus\mathcal{U}_2)$,
with $\mathcal{U}_i$ the projection of $U_i$ in $M_i$
(in other words as in  Definition \ref{definition-connectedsum},
$\mathcal{U}_i$ is an open set contained in $U_{M_i}$, containing $K_{M_i}$
and of boundary isotopic to $\partial K_{M_i}$).
We can assume, without loss of generality, the $U_i$ to coincide with the interior of their closure,
and that
there exists $g\in\PGL{3}$ such that
$U_2=g(U_1)$,
and anti-flag involutions $\varphi_1$ and $\varphi_2=g\varphi_1g^{-1}$
such that $\varphi_i(\partial U_i)=\partial U_i$
and $\varphi_i(\X\setminus\overline{U_i})=U_i$
(see  subsection \ref{subsection-preuvetheoremconnectedsum} for more details).
We emphasize that $\varphi_1$ and $\varphi_2$ are conjugate because any two anti-flag involutions
are conjugate ( see subsection \ref{subsection-dualinvolution}).
\par For $i=1,2$, $\Gamma_i$ acts freely and properly discontinuously on
\begin{equation}\label{equation-omegaiprime}
 \Omega_i'\coloneqq\Omega_i\setminus  \bigcup_{\gamma\in \Gamma_i}(\gamma\overline{U_{i}})
\subset\X
\end{equation}
with fundamental domain $D_i\setminus \overline{U_i}$,
and we will construct recursively an open subset $O$ of $\X$
as a tree-like gluing of $\Omega_1'$ and $\Omega_2'$
thanks to the involutions $\varphi_i$.
For the first step, we glue to $O_1\coloneqq\Omega_1'$ a copy of $\Omega_2'$
by attaching them through $g^{-1}\varphi_2=\varphi_1 g^{-1}$
on the components $\partial U_1$ and $\partial U_2$
of their boundaries.
Namely, since $\Omega_2'^{*}\coloneqq g^{-1}\varphi_2(\Omega_2')\subset U_1$
and $g^{-1}\varphi_2(\partial U_2)=\partial U_1$,
$O_2^{\id}=\Omega_1'\cup \overline{\Omega_2'^{*}}$
is an open subset of $\Omega_1$.
For any $\gamma\in\Gamma_1$,
we can attach on the same way the copy $\Omega_2^\gamma=\gamma(\Omega_2'^{*})\subset \gamma(U_1)$
of $\Omega_2'$
at the boundary component $\partial(\gamma U_{1})$
of $\Omega_1'$ to get an open set $O_2^{\gamma}$,
obtaining eventually the open subset $O_2\coloneqq \cup_{\gamma\in\Gamma_1}O_2^{\gamma}$
where each
``hole'' $\gamma\overline{U_{1}}$ has been ``filled in'' with the corresponding copy $\Omega_2^\gamma$
of $\Omega_2'$.
Note that $g^{-1}\varphi_2(U_2)=\X\setminus\overline{U_1}$
and thus with $\Omega_2^*\coloneqq g^{-1}\varphi_2(\Omega_2)$,
$\Omega_1'\subset\Omega_2^*$.
Therefore $\overline{O_1}\subset O_2$, $\overline{O_2}\subset\Omega_1\cap\Omega_2^*$
and $O_2\setminus\overline{O_1}=\cup_{\gamma\in\Gamma_1}\Omega_2^\gamma$
with $\Omega_2^\gamma\subset \gamma(U_1)$.
\par In the second step for any $\gamma_1\in\Gamma_1$ and $\gamma_2\in\Gamma_2\setminus\{\id\}$,
let $U_2(\gamma_1,\gamma_2)=\gamma_1g^{-1}\varphi_2\gamma_2(U_2)$
be the copy of $U_2$ contained in $\gamma_1(U_1)\setminus\Omega_2^{\gamma_1}$
and corresponding to $\gamma_2(U_2)$.
As we previously did for the copies of $\Omega_2'$ glued to $\Omega_1'$,
we can now glue to each created boundary component
$\partial U_2(\gamma_1,\gamma_2)$
of $O_2$ the suitable copy $\Omega_1^{\gamma_1,\gamma_2}$
of $\Omega_1'$ on $\partial U_1$ through $\varphi_1$.
More precisely, as before $\Omega_1'^*\coloneqq g\varphi_1(\Omega_1')\subset U_2$
and $\Omega_2'\cup\overline{\Omega_1'^*}\subset\Omega_2$ is open,
hence with $\Omega_1^{\gamma_1,\gamma_2}\coloneqq(\gamma_1g^{-1}\varphi_2\gamma_2)(\Omega_1'^*)$,
$O_3^{\gamma_1,\gamma_2}\coloneqq O_2\cup\Omega_1^{\gamma_1,\gamma_2}\subset\Omega_1\cap\Omega_2^*$
is open.
This leads to an open subset $O_3$ such that $\overline{O_2}\subset O_3$,
$\overline{O_3}\subset\Omega_1\cap\Omega_2^*$
and $O_3\setminus\overline{O_2}=\cup\Omega_1^{\gamma_1\gamma_2}$,
the union being taken on all $\gamma_1\in\Gamma_1$
and $\gamma_2\in\Gamma_2\setminus\{\id\}$.
We can then continue this procedure recursively
to obtain an increasing sequence $O_n$ of open sets,
and eventually define
$O=\cup_n O_n$ which is an open subset of $\Omega_1\cap\Omega_2^*$.
\par We now introduce the group $\Gamma_2'=(g^{-1}\varphi_2)\Gamma_2(g^{-1}\varphi_2)^{-1}$,
and emphasize that $\Gamma_2'\subset\PGL{3}$
since $\varphi_2=g\varphi_1g^{-1}$ is involutive and equivariant.
Since $(g^{-1}\varphi_2)^{-1}=\varphi_2g=g\varphi_1$,
$O$ is by construction $\Gamma_1$- and $\Gamma_2'$-invariant.
 Let us denote by
$O^{i}\coloneqq\cup_{k\equiv i [2]}O_k\setminus \overline{O_{k-1}}$
the ``$\Omega_i'$-part'' of $O$ (where $O_0=\varnothing$).
Then $\Gamma_1$ (resp. $\Gamma_2'$) preserves $O^{1}$
(resp. $O^{2}$).
Since $\Gamma_1$ (respectively $\Gamma_2'$) acts freely and properly discontinuously
on $\Omega_1$ (resp. on $\Omega_2^{*}$)
and $O\subset \Omega_1\cap \Omega_2^{*}$,
both groups act freely and properly discontinuously on $O$.
For the same reason, $\Gamma_1$ (respectively $\Gamma_2'$)
acts freely and properly discontinuously
on  $O^{1}$ (resp. on $O^{2}$),
and $\Gamma_1\backslash O^{1}$
(resp. $\Gamma_2\backslash O^{2}$)
is moreover canonically identified
with $M_1\setminus\overline{\mathcal{U}_1}$ (resp. $M_2\setminus\overline{\mathcal{U}_2}$)
with $\mathcal{U}_i$ the projection of $U_i$ in $M_i$.
\par A standard ping-pong-like argument shows now that the group
$\Gamma$ generated by $\Gamma_1$ and $\Gamma_2'$
is a free product of $\Gamma_1$ and $\Gamma_2'$, and
acts freely and properly discontinuously on $O$.
Indeed using the notation introduced previously,
for any $(\gamma_1,\gamma_2)\in\Gamma_1\times\Gamma_2$:
$\Omega_2^{\gamma_1}=\gamma_1(\Omega_2'^*)$ and
$\Omega_1^{\gamma_1,\gamma_2}=\gamma_1\gamma_2^*(\Omega_1')$
with $\gamma_2^*\coloneqq(g^{-1}\varphi_2)\gamma_2(g^{-1}\varphi_2)^{-1}$
(observe that $\Omega_1^{\gamma_1,\id}=\Omega_1'$).
More generally, continuing the recursive construction described above
with analog notations,
for any reduced word $w=g_1g_2\dots g_n$
whose letters are alternatively in $\Gamma_1$ and $\Gamma_2$,
and denoting by $\bar{w}$ the image of $w$ in $\Gamma$ obtained by replacing any $\gamma\in\Gamma_2$
by $\gamma^*$, we obtain
$\Omega_2^w=\bar{w}(\Omega_2'^*)$
and $\Omega_1^w=\bar{w}(\Omega_1')$.
The key-remark is now that by the very construction of $O$,
for any two distinct reduced words $w$ and $w'$ the subsets
$\Omega_1^w$, $\Omega_1^{w'}$, $\Omega_2^w$ and $\Omega_2^{w'}$
are pairwise disjoint.
This fact allows us to conclude,
by the same argument than the usual ping-pong lemma,
that the map $w\mapsto \bar{w}\in\Gamma$
sending a reduced word to its image induces an isomorphism between
the free product $\Gamma_1\star\Gamma_2$ and $\Gamma$,
and that the action of $\Gamma$ on $O$ is free and properly discontinuous.
\par In the end $M\coloneqq\Gamma\backslash O$ is a Kleinian flag manifold containing a copy
$M_i'$ of $M_i\setminus\overline{\mathcal{U}_i}$ and
such that $M=\overline{M_1'}\cup\overline{M_2'}$,
\emph{i.e.} $M$ is flag isomorphic to $S$ which concludes the proof.
\end{proof}

Note that the use of anti-flag involutions is crucial in the construction
of the open domain of discontinuity $O$ in the previous proof.
\par We now construct using Theorem \ref{theoremintro-sommeconnexe}
a large family of examples of Kleinian flag manifolds.
We use in the statement below the notation
from \cite{mion-mouton_geometrical_2022}.
\begin{proposition}\label{proposition-bouquetsSchottky}
Let $\Gamma\subset\PGL{3}$ be a Schottky subgroup of rank d
of separating handlebodies $\{H^-_i,H^+_i\}_{i=1}^d$.
 Let $M=\Gamma\backslash\Omega(\Gamma)$ be the associated Schottky flag manifold.
 Then for any $\alpha-\beta$ bouquet of circles $B$ contained in
 the fundamental domain $\X\setminus\cup_{i=1}^d(H^-_i\cup H^+_i)$,
 $B$ admits a neighbourhood which embedds in $M$.
 Theorem \ref{theoremintro-sommeconnexe} can thus be applied to
 glue $M$ along the resulting $\alpha-\beta$ bouquet of circles
 $B_M$ in $M$.
\end{proposition}
\begin{proof}
 This follows from the fact that $D=\X\setminus\cup_i(H^-_i\cup H^+_i)$
  is a fundamental domain for the action of $\Gamma$ on $\Omega(\Gamma)$
  according to \cite[Theorem D]{mion-mouton_geometrical_2022}.
  The restriction of the canonical projection $\Omega(\Gamma)\to M$
  to $D\supset B$ is thus a flag embedding,
  which proves the claim.
\end{proof}

\begin{remark}\label{remark-bouquetspasfermescompactification}
Observe that not every $\alpha$ (respectively $\beta$) leaf of a
Schottky flag manifold is closed.
For instance, the examples of \cite[\S 4.4]{mion-mouton_geometrical_2022}
are isomorphic up to a finite index to Schottky flag manifolds,
and are compactifications of $\Fiunitan{\Sigma}$ with $\Sigma$
a non-compact hyperbolic surface.
Here $\Fiunitan{\Sigma}$ is endowed with a natural flag structure whose $\alpha$
and $\beta$ leaves are the stable and unstable horocycle of the geodesic flow
(see \cite[\S 1 and Lemma 4.3]{mion-mouton_partially_2022}).
They are thus not closed, and \cite[Proposition 4.9]{mion-mouton_geometrical_2022} ensures
that some of them remain unclosed in the compactification.
\end{remark}

\begin{corollary}\label{corollary-newKleinianexamples}
 Let $M$ and $N$ be two Schottky flag manifolds,
 and $B_M$, $B_N$ be two $\alpha-\beta$ bouquets of circles
 contained in the respective fundamental domains of $M$ and $N$.
 Then there exists a flag surgery of $M$ and $N$
 along $B_M$ and $B_N$,
 which is a Kleinian flag manifold.
\end{corollary}
\begin{proof}
 This is a direct consequence of Proposition
 \ref{proposition-bouquetsSchottky}
 and Theorem \ref{theorem-Kleinian-connectedsum}.
\end{proof}

\begin{remark}\label{remark-nouveauxexemplesKleiniens}
 Comparing with \cite[Theorem C]{mion-mouton_geometrical_2022},
 we note that the flag surgery of two flag manifolds
 $M$ and $N$ along two $\alpha-\beta$ bouquets,
 is related to the procedure of removing two bouquets $B_1$
 and $B_2$ from the \emph{same} flag manifold
 and gluing together neighbourhoods of these bouquets.
 Comparing with topological surgeries,
 the former procedure is a flag structure realization
 of an amalgated product,
 while the latter one is analog to an HNN-extension.  In particular, one can obtain Schottky 
 flag manifolds by surgery of several of the examples $M_0$ (see subsection \ref{soussection-Schottkyungenerateur}).
 \end{remark}

\section{New non-Kleinian flag structures}
The goal of this  section is to construct new examples of non-Kleinian flag manifolds.
We describe in  subsection \ref{subsection-deformationstriviales}
a first very specific family of such examples, obtained as deformations of cyclic Schottky
flag manifolds on $\Sigma_2\times\Sn{1}$.
We use then surgeries to give a general recipe producing non-Kleinian examples,
and conclude in  subsection \ref{subsection-proofnonKleiniangluings}
the proof of Theorem \ref{theoremintro-structuresnonKleiniennesparsommeconnexe}.

\subsection{First example: deformations of the flag structure on $M$}
\label{subsection-deformationstriviales}
The first examples of non-Kleinian flag structures arise as deformations of the examples of subsection \ref{soussection-Schottkyungenerateur} which have infinite cyclic holonomy.
These examples were suggested to us by C.  Frances 
and are analogous to the construction of
affine structures on the torus with non-discrete holonomy
(\cite[\S 6 p. 79]{gunning_special_1967}).
Consider a loxodromic element $g\in\PGL{3}$
as in subsection  \ref{soussection-Schottkyungenerateur}.
We let again  $x_-,x_+\in\X$ be the   repelling and  attracting points
and $B^-=B_{\alpha\beta}(x_-)$, $B^+=B_{\alpha\beta}(x_+)$
the associated $\alpha-\beta$ bouquet of circles.
Define $\Omega=\X\setminus(B^-\cup B^+)$ and
$\Gamma\coloneqq\langle g \rangle$.
Then $M_0=\Gamma\backslash\Omega$ (see subsection \ref{soussection-Schottkyungenerateur}) is a Kleinian flag manifold diffeomorphic to
$\Sigma_2\times\Sn{1}$.
\begin{proposition}\label{proposition-deformationsM}
 There exists on $M_0$
 distinct
 pairwise non-isomorphic flag structures,
 which are continuous deformations of the Kleinian structure of $\Gamma\backslash\Omega$,
 but which are not virtually Kleinian.
\end{proposition}
\begin{proof}
The fundamental group of the quotient manifold $M_0$ is $\pi_1(M_0)=\pi_1(\Sigma_2)\times\Z$.
The holonomy map $\rho\colon\piun{M_0}\to \PGL{3}$ is defined on generators
$a_1,b_1,a_2,b_2, z$,
where $a_1,b_1,a_2,b_2$ are standard generators of $\pi_1(\Sigma_2)$
subject to the single relation
\[
[a_1,b_1][a_2,b_2]=\id,
\]
and
$z$ is the positive generator of $\Z$,
by $\rho(a_i)=\rho(b_i)=id$, $1\leq i\leq 2$, and $\rho(z)=g$.
Let  $x=(s_1,s_2,t_1,t_2)\in\R^4$
be any choice of four real numbers such that
$\Z s_1+\Z s_2+\Z t_1+\Z t_2$
is dense in $\R$.
Note that this choice can be made with
$\norme{x}_1=\abs{s_1}+\abs{s_2}+\abs{t_1}+\abs{t_2}$
as small as one wants.
Since the one-parameter group $\{g^t\}_{t\in\R}$ is abelian,
the relations
\begin{itemize}
 \item  $\rho_x(a_1)=g^{s_1} $ and $\rho_x(a_2)=g^{s_2}$,
 \item $\rho_x(b_1)=g^{t_1}$ and $\rho_x(b_2)=g^{t_2}$,
 \item $\rho_x(z)=g$,
\end{itemize}
define a unique morphism $\rho_x\colon\piun{M_0}\to\PGL{3}$.
According to the Ehresmann-Thurston principle (see \cite{thurston_three-dimensional_1997}
and \cite{canary_notes_1987}),
for any small enough $\norme{x}_1$,
$\rho_x$ is close enough to $\rho$
to be the holonomy morphism
of a flag structure $\Lm_x$ on $M_0$
which is close to the original Kleinian one.
In particular, for $\norme{x}_1$ small enough the resulting flag structure is
thus a continuous deformation of the original Kleinian structure of
$\Gamma\backslash\Omega$
(since the moduli space of $(G,X)$-structures is locally arwise-connected
for any homogeneous space $X$ under a connected Lie group $G$,
see for instance \cite[Deformation theorem p.178]{goldman_geometric_1988}).
Note furthermore that there exists choices
of $x$ with $\norme{x}_1$ as small as one wants, and
whose associated sets of traces
$\enstq{\tr(g^{js_1+ks_2+lt_1+mt_2+n})}{(j,k,l,m,n)\in\Z^5}$ are
pairwise distinct.
These sets, being conjugacy invariants  of the representation $\rho_x$ in $\PGL{3}$,
 are  invariants of the associated flag structures $\Lm_x$,
which are thus pairwise non-isomorphic.
Lastly, the holonomy group of any finite-index covering of the flag manifold
$(M_0,\Lm_x)$ is a finite-index subgroup of $\Image(\rho_x)$,
and is thus non-discrete since $\Image(\rho_x)$
is dense by assumption.
In particular, $\Lm_x$ has  no finite-index Kleinian covering.
\end{proof}

\subsection{A suitable covering}\label{soussection-suitablecovering}
Recall that $\Omega=\X\setminus(B^-\cup B^+)$
denotes the image of the developing map $\delta_{M_0}$
as in subsection \ref{soussection-Schottkyungenerateur}.
The quotient manifold is $M_0=\Gamma_0\backslash\Omega$,
where $\rho_0 (\pi_1(M_0))=\Gamma_0\coloneqq\langle g \rangle$
is the holonomy group, and $B^-$ and $B^+$ are the  repelling and  attracting bouquets of circles.
We also denote by $E\simeq\tilde{\Omega}$ the universal cover of $M_0$,
and by
$\delta_{M_0}\colon E\to\Omega=\delta_{M_0}(E)$ the developing
map of $M_0$ which is just the universal covering map of $\Omega$.
\par Let $B$ be a $\alpha-\beta$ bouquet of circles
disjoint from $B^-$ and $B^+$.
Then, possibly replacing $g$ by some big enough power
and $H$ by a smaller neighbourhood,
we can assume that $B$ is contained in the fundamental domain $D$.
There exists then an open neighbourhood $\mathcal{U}_0\subset\X$
of $B$ contained in $D$,
and with $\pi_{\Gamma}\colon\Omega\to M_0$ the canonical projection,
$\pi_{\Gamma}\restreinta_{\mathcal{U}_0}\colon\mathcal{U}_0\to
U_0\coloneqq\pi_{\Omega}(\mathcal{U}_0)$ is thus an embedding
of $\mathcal{U}_0$ in $M_0$ of image $U_0$,
neighbourhood of the $\alpha-\beta$ bouquet of circles $B_{M_0}=\pi_\Gamma(B)$.
\par The goal of this subsection is to show the following:
\begin{lemma}\label{lemma-revetementM}
 There exists a Galoisian order 2 covering $F\colon M\to M_0$
such that with $\pi_{M}\colon E\to M$ the universal covering map of $M$,
there exists a connected component $U$ of $F^{-1}(U_0)$ such that:
\begin{enumerate}
 \item $F\restreinta_{U}\colon U\to U_0$ is a diffeomorphism;
 \item $\delta_M(E\setminus\pi_{M}^{-1}(U))=\delta_M(E)=\Omega$.
\end{enumerate}
\end{lemma}

We will construct this covering from a suitable covering $\Sigma_3$
of $\Sigma$ by a genus 3 surface,
and to this end we first have to homotope
the bouquet of circles $B$ to $\Sigma$.

\subsubsection{Homotopy of $B$ to $\Sigma$}
In
 \eqref{equation-definitionDdomainefondamental},
we described a fundamental domain $D$ for the action of a
loxodromic element $g$ on the flag space $\X$.
This fundamental domain is naturally identified with
a product $\Sigma\times (0,1)$,
with $\Sigma$ a genus two surface
which is the boundary of
a tubular neighbourhood
of the repelling bouquet $B^-$ of $g$.
The quotient space $M_0=\Gamma_0\backslash\Omega$ is a flag manifold
homeomorphic to $\Sigma\times\Sn{1}$
with $\Sigma$ a genus 2 closed connected and orientable surface.

An explicit tubular neighbourhood of a bouquet can be constructed as follows.
Consider the two $\PGL{3}$-equivariant fiber-bundle projections of $\X$, $\pi_\alpha$ and $\pi_\beta$,
which are the  first and second coordinate projections
onto $\RP{2}$ and $\RP{2}_*$.   The 
\emph{$\alpha-\beta$ bouquet}
passing through $x$, $\Balphabeta(x)=\Calpha(x)\cup\Cbeta(x)$, is the union of the two fibers $\Calpha(x)= \pi^{-1}_\alpha(\pi_\alpha(x))$ and 
$\Cbeta(x)= \pi^{-1}_\beta(\pi_\beta(x))$.  

A convenient description of a tubular neighbourhood of $\Balphabeta(x)$ is given as the union of fibered neighbourhoods
of each of the  circles of the bouquet.
Let $U_{\alpha}$ and  $U_{\beta}$ be two neighbourhoods of $\pi_\alpha(x)$ and
$\pi_\beta(x)$ respectively, which we suppose to be homeomorphic to discs. The fibrations are trivial over each neighbourhood and, therefore, $\pi^{-1}_\alpha(U_{\alpha})$ and
$\pi^{-1}_\beta(U_\beta)$ are tubular neighbourhoods of the two circles forming the bouquet.  Each neighbourhood is a fibered full torus.

\begin{lemma}
Let $\Balphabeta(x)$ be the bouquet of circles associated to the flag $x$.    Then, for any neighbourhoods $U_{\alpha}$ and  $U_{\beta}$  in $\RP{2}$ and $\RP{2}_*$ homeomorphic to discs as above, $\pi^{-1}_\alpha(U_{\alpha})\cup \pi^{-1}_\beta(U_\beta)$ is a tubular neighbourhood of $\Balphabeta(x)$.
\end{lemma}

\begin{proof}
Clearly, $\pi^{-1}_\alpha(U_{\alpha})$ and 
$\pi^{-1}_\beta(U_\beta)$ are tubular neighbourhoods of the two circles forming the bouquet.  We note
that the intersection of these two neighbourhoods is homeomorphic to a ball in a special case.
Consider an affine chart containing $U_{\alpha}$,
which we can assume to be a disc centered at the origin.
The flag $x$ can be identified to the pair consisting of the origin and the $x$-axis.
The neighbourhood $U_\beta$ can be chosen to be formed of parallel lines of slopes ranging,
for small $\epsilon$,
from $-\epsilon$ to $\epsilon$, and  intersecting the disc $U_\alpha$.
The intersection $\pi^{-1}_\alpha(U_{\alpha})\cap \pi^{-1}_\beta(U_\beta)$
is then a cylinder parametrized by $U_\alpha\times (-\epsilon,\epsilon)$.
\end{proof}

From this lemma we conclude that any two bouquets $\Balphabeta(x_i)$, $i=1,2$,
with $\pi_\alpha(x_i)\in U_{\alpha}$ and $\pi_\beta(x_i)\in U_\beta$ are homotopic.
This implies that one can choose
a bouquet $B\subset \Sigma\times (0,1)$ contained in the fundamental domain
and it will be homotopic to a bouquet contained in the surface $\Sigma$.

The goal now is to construct a double cover of the quotient  space
$M_0=\Gamma_0\backslash\Omega=\Sigma_2\times\Sn{1}$
such that the bouquet $B$ lifts to two bouquets,
each of them homeomorphic to the original one by the covering map.  It is sufficient to
construct a double cover of the surface with the same property as $B$
can be deformed to the surface $\Sigma$.

\subsubsection{Suitable covering of $\Sigma$}
Let  $\Sigma$ be a genus two closed connected and orientable surface
and $a_1,b_1, a_2,b_2$ be standard generators of $\pi_1(\Sigma)$.
The bouquet $B$ can be identified, through a homotopy, to the element $a_1\cup a_2$ in $\Sigma$.
We need now the following :

\begin{lemma}\label{lemma-coveringsurface}
Let  $\Sigma$ be a genus two surface and $a_1,b_1, a_2,b_2\subset \Sigma $ be standard generators of $\pi_1(\Sigma)$.
Then, there exists a Galoisian double cover by a genus three surface $\pi : \Sigma_3\to \Sigma$,
such that the inverse image of  $a_1 a_2$ has two connected components.  In other words, the covering exact sequence
$$
\{e\}\to \pi_1(\Sigma_3)\to \pi_1(\Sigma)\to \Z/2\Z\to \{0\}
$$
satisfies $a_1,a_2\in \ker \phi$, where $\phi : \pi_1(\Sigma)\to \Z/2\Z$ is the quotient map.

\end{lemma}
\begin{proof}
 This lemma is straightforward once we choose $\pi_1(\Sigma_3)$ to be the kernel of the
homomorphism $\pi_1(\Sigma)\to \Z/2\Z$ given by $a_1\to 0,a_2\to 0,b_1\to 1,b_2\to 1$.
\end{proof}

\subsubsection{Proof of Lemma \ref{lemma-revetementM}} 

We can now prove Lemma \ref{lemma-revetementM}.
Consider the double cover $\pi : \Sigma_3\to \Sigma$ of Lemma
\ref{lemma-coveringsurface}.
Recall that $M_0\simeq\Sigma\times\Sn{1}$ and define the double cover $M=\Sigma_3\times\Sn{1}$
$$
F : M\to M_0
$$
induced by the cover $\pi$.  Let $U_0$ be a tubular neighbourhood of the bouquet $B$ (which, by homotopy,  we may consider to be the
union of the two generators $a_1$ and $a_2$ in $\Sigma$).  By  lemma
\ref{lemma-coveringsurface}, $F^{-1}(U_0)$ has two connected components $U$ and $U'$, each of them homeomorphic to $U_0$ through $F$.  

Consider now the universal covering 
$\pi_{M}\colon E\to M$
and the developing map $\delta_{M}:E\to \Omega$.
The latter coincides
with $\delta_{M_0}\colon E\to \Omega$ since $M$ is a cover of $M_0$,
and satisfies thus $\delta_M(\gamma x)=\rho_{M_0}(\gamma)\delta_M(x)$
for any $\gamma\in \pi_1(M_0)$ and $x\in E$.
Since
$\rho_{M_0}$ restricted to $\pi_1(\Sigma)$ is trivial
according to \eqref{equation-rhoM0}, we have thus
\begin{equation}\label{equation-relationpi1Sigma}
\delta_M(\gamma x)=\delta_M(x)
\end{equation}
for all $\gamma\in \pi_1(\Sigma)$. Choose one of the connected components $U\subset F^{-1}(U_0)$. 
Observe that $\pi_{M}$ restricts to a covering map $E\setminus \pi_{M}^{-1}(U)\to M\setminus U$  and the developing map defined on 
the universal cover of ${M\setminus U}$
descends to a map defined on $E\setminus \pi_{M}^{-1}(U)$.
Moreover if $x\in \pi_{M}^{-1}(U)$,
then since the covering $F : M\to M_0$ is Galois
there exists an element $\gamma\in \pi_1(\Sigma)$
such that
$\gamma x\in \pi_{M}^{-1}(U')\subset E\setminus \pi_{M}^{-1}(U)$,
hence
$\delta(x)=\delta(\gamma x)\in\delta_M(E\setminus\pi_{M}^{-1}(U))$
according to \eqref{equation-relationpi1Sigma}.
We obtain thus that
$\delta_M(E\setminus\pi_{M}^{-1}(U))=\delta_{M}(E)=\Omega$,
which concludes the proof of the lemma.

\subsection{Conclusion of the proof of
Theorem \ref{theoremintro-structuresnonKleiniennesparsommeconnexe}}
\label{subsection-proofnonKleiniangluings}
Let $N$ be a  Kleinian flag manifold
such that \begin{enumerate}
 \item there exists
 $B_N\subset N$ an $\alpha-\beta$ bouquet
 of two circles admitting a neighbourhood $U_N$
flag isomorphic to the neighbourhood of an $\alpha-\beta$
bouquet of circles in $\X$;
 \item denoting $\Gamma_N\subset\PGL{3}$ the holonomy group of $N$,
 there exists a loxodromic element $h\in \Gamma_N$.
\end{enumerate}
We now use the open set $U_M\coloneqq U\subset M$
given by Lemma \ref{lemma-revetementM},
which is a neighbourhood of the $\alpha-\beta$ bouquet of circles $B_M$,
and the neighbourhood $U_N\subset N$,
to form the flag surgery $S$ of $M$ and $N$
along $U_M$ and $U_N$
(see the proof of Theorem \ref{theoremintro-sommeconnexe}
in subsection \ref{subsection-preuvetheoremconnectedsum}).
Let $\delta_S\colon\tilde{S}\to\X$
and $\rho_S\colon\piun{S}\to\PGL{3}$
be the developing map and holonomy morphism of $S$,
and $\Gamma_S=\rho_S(\piun{S})$ its holonomy group.  Recall that $E$ is the universal cover of $M$.

\begin{lemma}\label{lemma-imagedeveloppantesommeconnexe}
Possibly composing $\delta_S$ with an element of $\PGL{3}$,
we have
  $\delta_M(E\setminus\pi_M^{-1}(U))\subset\delta_S(\tilde{S})$.
\end{lemma}
\begin{proof}
According to the Definition \ref{definition-connectedsum}
of a flag surgery, there exists a flag embedding
$j\colon M\setminus U\to S$.
More precisely,
the proof of Theorem \ref{theorem-Kleinian-connectedsum} shows that
one
can obtain a cover $Y$ of $S$ by taking the tree formed by
the components $E\setminus\pi_M^{-1}({U_M})$
and $\tilde{N}\setminus\pi_N^{-1}({U_N})$ glued
accordingly. Note that contrary to the case of
Theorem \ref{theorem-Kleinian-connectedsum},
$Y$ is not embedded in $\X$,
however the developing map $\delta\colon\tilde{S}\to\X$ descends
to a flag morphism $\bar{\delta}_S$ from $Y$ to $\X$.
In particular, we can choose for $\delta_S$ the unique developing map
of $S$ such that
$\bar{\delta}_S\restreinta_{E\setminus\pi_M^{-1}({U_M})}$
coincides with $\delta_M\restreinta_{E\setminus\pi_M^{-1}({U_M})}$,
so that
$\delta_M(E\setminus\pi_M^{-1}(\bar{U}))\subset\delta_S(\tilde{S})$
as claimed.
\end{proof}

According to Lemmas \ref{lemma-revetementM} and \ref{lemma-imagedeveloppantesommeconnexe},
$\delta_S(\tilde{S})$ contains $\Omega$.  We denote by $\Omega_N\subset \X$ the domain of discontinuity 
of the holonomy $\Gamma_N$. 
\par Now, following the proof of Theorem \ref{theorem-Kleinian-connectedsum}
we construct a tree-like manifold $Y$ which is a cover of $S$
by attaching $E$ 
(using involutions as described in subsection \ref{subsection-definitionexemples})
along each component $ \partial \gamma\pi_M^{-1}{U}_{M}$, with  $\gamma\in \Gamma$, to 
$\Omega_N\setminus  \bigcup_{\gamma\in \Gamma_N}(\gamma{U'_{N}})$ (here ${U'_{N}}\subset \Omega_N$ is a lift of $U_N$).  We repeat the attaching procedure recursively as in Theorem \ref{theorem-Kleinian-connectedsum}.
  Remark that in this case, contrary to the surgery of two Kleinian structures,
  the manifold $Y$ is not embedded into $\X$.  But the developing map defined on $\tilde {S}$,
  descends to a map defined on $Y$. The holonomy group contains the group $\Gamma$
  and a subgroup isomorphic to $\Gamma_N$ as in the proof of Theorem \ref{theorem-Kleinian-connectedsum}. This can be checked by observing that the developing map defined on $Y$ is equivariant with respect to the action of these groups.

By construction, the neighbourhood $ \pi_M^{-1}{U}_{M}\subset \X$
contains a bouquet in the limit set of the
holonomy of the surgery.
In fact, this is the case for each $\gamma {U'_{N}}$
with $\gamma\in \Gamma$.
Indeed, $\X\setminus {U'_{N}}$ contains a bouquet in the limit set
of the holonomy of $N$ and by the appropriate inversion used in the surgery,
it will be translated to inside the neighborhood $\gamma\pi_M^{-1}{U}_{M}$.
The conclusion follows from Lemma \ref{lemma-KulkarniPinkall}.

\bibliographystyle{alpha}
\bibliography{biblio-varietesdrapeauxnonKleiniennes}

\end{document}